\documentclass[reqno]{amsart}

\usepackage{amscd,amsmath,amssymb,amsfonts}
\usepackage{mathrsfs}
\usepackage[dvips]{color} 
\usepackage[all]{xy}

\numberwithin{equation}{section}

\theoremstyle{plain}
    \newtheorem{thm}{Theorem}[section]
    \newtheorem{lem}[thm]{Lemma}

    \newtheorem{cor}[thm]{Corollary}

    \newtheorem{conj}[thm]{Conjecture}
\theoremstyle{definition}    
    \newtheorem{defn}[thm]{Definition}
    
    \newtheorem{rem}[thm]{Remark}

\def\Coker{\mathrm{Coker}}

\def\dR{{\mathrm{d\hspace{-0.2pt}R}}}            

\def\Ext{{\mathrm{Ext}}}

\def\ker{{\mathrm{Ker}}}          
\def\Ker{{\mathrm{Ker}}}          

\def\Pic{{\mathrm{Pic}}}

\def\NS{{\mathrm{NS}}}

\def\pr{{\mathrm{pr}}}

\def\pr{{\mathrm{pr}}}
\def\reg{{\mathrm{reg}}}          %
\def\Res{\mathrm{Res}}
\def\Gr{{\mathrm{Gr}}}
\def\rank{{\mathrm{rank}}}
\def\zar{{\mathrm{zar}}}
\def\bA{{\mathbb A}}

\def\C{{\mathbb C}}

\def\P{{\mathbb P}}

\def\Q{{\mathbb Q}}

\def\Z{{\mathbb Z}}

\def\cH{{\mathscr H}}
\def\cM{{\mathscr M}}
\def\cD{{\mathscr D}}

\def\cF{{\mathscr F}}
\def\cJ{{\mathscr J}}

\def\O{{\mathscr O}}

\def\cU{{\mathscr U}}
\def\vg{\varGamma}

\def\lra{\longrightarrow}

\def\hra{\hookrightarrow}

\def\ot{\otimes}
\def\op{\oplus}
\def\wt#1{\widetilde{#1}}

\def\ol#1{\overline{#1}}

\def\os#1#2{\overset{#1}{#2}}

\def\Aut{\mathrm{Aut}}

\begin{document}
\title{A functional logarithmic formula for hypergeometric functions ${}_3F_2$.}
\author{Masanori Asakura and Noriyuki Otsubo}
\address{Department of Mathematics, Hokkaido University, Sapporo, 060-0810 Japan}
\email{asakura@math.sci.hokudai.ac.jp}
\address{Department of Mathematics and Informatics, Chiba University, Chiba, 263-8522 Japan}
\email{otsubo@math.s.chiba-u.ac.jp}
\date{2016}

\maketitle

\section{Introduction}\label{intro-sect}
For $\alpha_i,\beta_j\in\C$ with $\beta_j\not\in\Z_{\leq 0}$,
the {\it generalized hypergeometric function} is defined by a power series expansion
\[
{}_pF_{p-1}\left({\alpha_1,\cdots,\alpha_p\atop \beta_1,\cdots,\beta_{p-1}};x\right)
=\sum_{n=0}^\infty\frac{(\alpha_1)_n\cdots(\alpha_p)_n}{(\beta_1)_n\cdots(\beta_{p-1})_n}
\frac{x^n}{n!},
\]
where 
\[
(\alpha)_0:=1,\quad (\alpha)_n:=\alpha(\alpha+1)\cdots(\alpha+n-1)\,\mbox{ for }n\geq1
\]
denotes the Pochhammer symbol.
When $p=2$, this is called the Gauss hypergeometric function.
This has the analytic continuation to $\C$, and then
becomes the multi-valued function which is locally holomorphic on $\C\setminus\{0,1\}$.
A number of formulas are discovered since 19th century 
(e.g. \cite{NIST} Chap. 15,16), and they are applied in lots of areas in mathematics.
In the present,
 the hypergeometric function is one of the most important tool in mathematics.

In \cite{a-o-t}, we discussed the special values of 
${}_3F_2\left({1,1,q\atop a,b};x\right)$ at $x=1$,
and gave a sufficient conditions for that it is a $\ol\Q$-linear combination
of log of algebraic numbers, namely
\[
{}_3F_2\left({1,1,q\atop a,b};1\right)
\in
\ol{\Q}+\ol{\Q}\log\ol{\Q}^\times\\
:=\left\{a+\sum_{i=1}^nb_i\log c_i\mid a,b_i,c_i\in\ol{\Q},\,
c_i\ne0,\,n\in\Z_{\geq0}\right\}.
\]
The goal of this paper is to give its functional version.
To be precise, set
\[
\ol{\Q(x)}+\ol{\Q(x)}\log\ol{\Q(x)}^\times\\
:=\left\{f+\sum_{i=1}^ng_i\log h_i\mid f,g_i,h_i\in\ol{\Q(x)},\,
h_i\ne0,\,n\in \Z_{\geq0}\right\}
\]
where $\ol{\Q(x)}$ denotes the algebraic closure of the field of rational functions $\Q(x)$.
We say the {\it logarithmic formula} holds for a function $f(x)$ if 
it belongs to the above set.
The main theorem
is to give a sufficient condition on $(a,b,q)$ for that the log formula holds
for the hypergeometric function ${}_3F_2\left({1,1,q\atop a,b};x\right)$.
Recall that two proofs are presented in \cite{a-o-t}, one uses
hypergeometric fibrations and the other uses Fermat surfaces.
In this paper we follow the method of hypergeometric fibrations, while
a new ingredient is employed from \cite{a-o-3}.
It seems impossible to prove the functional log formula according to
the method of Fermat.

\medskip

By developing the technique here, we can get {\it explicit} log formula in some cases.
For example, let
\[
e_1(x):=\frac{1}{2}
+x^{-\frac{1}{3}}
\left(-\frac{1}{4}+\frac{x}{8}+\frac{1}{4}\sqrt{1-x}\right)^{\frac{1}{3}}
+x^{-\frac{1}{3}}
\left(-\frac{1}{4}+\frac{x}{8}-\frac{1}{4}\sqrt{1-x}\right)^{\frac{1}{3}}
\]
\[
e_2(x):=\frac{1}{2}
+e^{-2\pi i/3}x^{-\frac{1}{3}}
\left(-\frac{1}{4}+\frac{x}{8}+\frac{1}{4}\sqrt{1-x}\right)^{\frac{1}{3}}
+e^{2\pi i/3}x^{-\frac{1}{3}}
\left(-\frac{1}{4}+\frac{x}{8}-\frac{1}{4}\sqrt{1-x}\right)^{\frac{1}{3}}
\]
\[
e_3(x):=\frac{1}{2}
+e^{2\pi i/3}x^{-\frac{1}{3}}
\left(-\frac{1}{4}+\frac{x}{8}+\frac{1}{4}\sqrt{1-x}\right)^{\frac{1}{3}}
+e^{-2\pi i/3}x^{-\frac{1}{3}}
\left(-\frac{1}{4}+\frac{x}{8}-\frac{1}{4}\sqrt{1-x}\right)^{\frac{1}{3}}
\]
\[
p_\pm=p_\pm(x):=\left(\frac{1\pm\sqrt{1-x}}{\sqrt{x}}\right)^{\frac{2}{3}},\quad 
q_j=q_j(x):=\frac{1-\sqrt{3x}\cdot e_j(x)}{1+\sqrt{3x}\cdot e_j(x)}.
\]
Then
\[
{}_3F_2\left({1,1,\frac{1}{2}\atop
\frac{7}{6},\frac{11}{6}};x\right)=
\frac{5\sqrt{3}}{36}x^{-\frac{1}{2}}
\left[
(p_++p_-)\log \left(\frac{q_1}{q_2}\right)+(e^{\frac{\pi i}{3}}p_++e^{-\frac{\pi i}{3}}p_-)
\log \left(\frac{q_2}{q_3}\right)\right].
\]
However there remains technical difficulties arising from algebraic cycles
to obtain explicit log formula in more general cases.

\section{Main Theorem}\label{main-sect}
Let 
$\hat{\Z}=\varprojlim_n\Z/n\Z$ be the completion, and
$\hat{\Z}^\times=\varprojlim_n(\Z/n\Z)^\times$ the group of units.
The ring $\hat{\Z}$ acts on the additive group $\Q/\Z$ in a natural way, and
then it induces $\hat{\Z}^\times\cong \Aut(\Q/\Z)$.
We denote by $\{x\}:=x-\lfloor x\rfloor$ the fractional part of $x\in\Q$.
The map $\{-\}:\Q\to[0,1)$ factors through $\Q/\Z$, which we denote by
the same notation.

\begin{thm}[Logarithmic Formula]\label{log}
Let $q,a,b\in \Q$ satisfy that none of $q,a,b,q-a,q-b,q-a-b$ is an integer.
Suppose
\begin{equation}\label{log-eq1}
1=\{sa\}+\{sb\}
+2\{-sq\}-\{s(a-q)\}-\{s(b-q)\}
\end{equation}
\[
(\Longleftrightarrow\quad
\min(\{sa\},\{sb\})<\{sq\}<\max(\{sa\},\{sb\}))
\]
for $\forall s\in \hat{\Z}^\times$.
Then
\[
{}_3F_2\left({n_1,n_2,q\atop a,~b};x\right)
\in\ol{\Q(x)}+\ol{\Q(x)}\log\ol{\Q(x)}^\times
\]
for any integers $n_i>0$.
\end{thm}
As we shall see in \S \ref{proof-sect},  one can shift the indices $n_i$, $q,a,b$ by arbitrary
integers by applying differential operators.
Thus it is enough to prove the log formula for ${}_3F_2\left({1,1,q\atop a,~b};x\right)$.

\medskip

Recall the main theorem of \cite{a-o-t} which asserts that
if
\begin{equation}\label{log-eq2}
2=\{sq\}+\{s(a-q)\}+\{s(b-q)\}+\{s(q-a-b)\}
\end{equation}
for $\forall s\in \hat{\Z}^\times$, then
\[
{}_3F_2\left({1,1,q\atop a,~b};1\right)
\in\ol{\Q}+\ol{\Q}\log\ol{\Q}^\times
\]
as long as it converges ($\Leftrightarrow$ $a+b>q+2$).
It is easy to see \eqref{log-eq1} $\Rightarrow$ \eqref{log-eq2}
while the converse is no longer true (e.g. $(a,b,q)=(1/6,1/4,1/2)$).
Theorem \ref{log} does not cover all of the main theorem of \cite{a-o-t}.
\begin{conj}[cf. \cite{a-o-t} Conjecture 5.2]
The converse of Theorem \ref{log} is true.
\end{conj}

In the seminal paper \cite{BH}, Beukers and Heckman gave a necessary and sufficient condition
for that ${}_pF_{p-1}$ is an algebraic function, or equivalently its monodromy group is finite.
Let $a_i,b_j\in\Q$.
Then their theorem tells that, 
under the condition that $\{a_i\}\ne\{b_j\}$ and $\{a_i\}\ne0$,
\[
{}_pF_{p-1}\left({a_1,\ldots,a_p\atop b_1,\ldots,b_{p-1}};x\right)\in\ol{\Q(x)}
\]
if and only if $(\{sa_1\},\ldots,\{sa_p\})$ and $(0,\{sb_1\},\ldots,\{sb_{p-1}\})$ interlace
for all $s\in \hat\Z^\times$
(loc.cit. Theorem 4.8).
Here we say that
two sets $(\alpha_1,\ldots,\alpha_p)$ and $(\beta_1,\ldots,\beta_p)$ interlace if and only if
\[
\alpha_1<\beta_1<\cdots<\alpha_p<\beta_p\mbox{ or }
\beta_1<\alpha_1<\cdots<\beta_p<\alpha_p
\]
when ordering $\alpha_1<\cdots<\alpha_p$ and $\beta_1<\cdots<\beta_p$.
In this terminology, 
\eqref{log-eq1} is translated into that $\{sq\}$ and $(\{sa\},\{sb\})$ interlace.
Our main theorem \ref{log} has no intersection with their theorem, while
they are obviously comparable.

\section{Hypergeometric Fibrations}\label{HG-sect}
\subsection{Definition}\label{HG-sect1}
Let $R$ be a finite-dimensional semisimple $\Q$-algebra.
Let $e:R\to E$ be a projection onto a number field $E$.
Let $X$ be a smooth projective variety over $k_\dR$, and $f:X\to \P^1$
a surjective map endowed with a multiplication on $R^1f_*\Q|_U$ by $R$
where $U\subset \P^1$ is the maximal Zariski open set such that $f$ is smooth over $U$. 
We say $f$ is a {\it hypergeometric fibration with multiplication by $(R,e)$} (abbreviated HG
fibration) if
the following conditions hold.
We fix an inhomogeneous coordinate $t\in\P^1$.
\begin{enumerate}
\item[\bf (a)]
$f$ is smooth over $\P^1\setminus\{0,1,\infty\}$,
\item[\bf (b)]
$\dim_E (R^1f_*\Q)(e)=2$ where we write $V(e):=E\ot_{e,R}V$ the $e$-part,
\item[\bf (c)]
Let $\Pic_f^0\to \P^1\setminus\{0,1,\infty\}$ be the Picard fibration whose general fiber
is the Picard variety $\Pic^0(f^{-1}(t))$,
and let $\Pic_f^0(e)$ be the component associated to the $e$-part $(R^1f_*\Q)(e)$
(this is well-defined up to isogeny).
Then $\Pic_f^0(e)\to\P^1\setminus\{0,1,\infty\}$ has a totally degenerate semistable reduction at $t=1$.
\end{enumerate}
The last condition {\bf (c)} is equivalent to say that
the local monodromy $T$ on $(R^1f_*\Q)(e)$ at $t=1$ is unipotent
and the rank of log monodromy $N:=\log(T)$ is maximal, namely 
$\rank(N)=\frac{1}{2}\dim_\Q (R^1f_*\Q)(e)$
($=[E:\Q]$ by the condition {\bf(b)}).

\subsection{HG fibration of Gauss type}\label{HG-sect2}
Let $f:X\to \P^1$ be a fibration over $\ol\Q$
whose general fiber $f^{-1}(t)$ is a nonsingular projective model of an affine curve
\begin{equation}\label{HG-sect1-eq1}
y^N=x^a(1-x)^b(1-tx)^{N-b},\quad 0<a,b<N,\, \mathrm{gcd}(N,a,b)=1.
\end{equation} 
$f$ is smooth outside $\{0,1,\infty\}$ so that the condition {\bf(a)} is satisfied.
The group $\mu_N$ of $N$-th roots of unity acts on $f^{-1}(t)$ by
$(x,y,t)\mapsto(x,\zeta y,t)$ for $\zeta\in\mu_N$, which gives rise to
 a multiplication on $R^1f_*\Q$ by the group ring $R_0:=\Q[\mu_N]$.
\begin{lem}\label{gauss.cond}[\cite{a-o}, Proposition 3.1]
Let $e_0:R_0:=\Q[\mu_N]\to E_0$ be a projection onto a number field $E_0$.
Then $(R_0,e_0)$ satisfies the conditions {\bf(b)} and {\bf(c)} if and only if
$ad\not\equiv0$ and $bd\not\equiv0$ modulo $N$ where 
$d:=\sharp\ker[\mu_N\to R_0^\times\os{e_0}{\to} E_0^\times]$.
\end{lem}

\begin{defn}\label{gauss.type}
We say that $f$ 
is a {\it HG fibration of Gauss type with multiplication by $(\Q[\mu_N],e)$}
if
$ad\not\equiv0$ and $bd\not\equiv0$ modulo $N$.
\end{defn}
Let $\chi:R_0\to \ol\Q$ be a homomorphism of $\Q$-algebra factoring through $e$.
Let $n$ be an integer such that $\chi(\zeta)=\zeta^{-n}$ for all $\zeta\in \mu_N$. 
Note $\gcd(n,N)=1$.
By \cite{archinard} (13), p.917,
$H^1_\dR(X_t)(\chi)\cap H^{1,0}$ is spanned by a 1-form
\[
\omega_n:=\frac{x^{a_n}(1-x)^{b_n}(1-tx)^{c_n}}{y^n}dx,
\]
\[
a_n:=\lfloor\frac{an}{N}\rfloor,\,
b_n:=\lfloor\frac{bn}{N}\rfloor,\,
c_n:=\lfloor\frac{Nn-bn}{N}\rfloor=n-b_n-1.
\]
Let $P_1$ (resp. $P_2$) be a point of $X_t$ above $x=0$ (resp. $x=1$).
There are $\mathrm{gcd}(N,a)$-points above $x=0$
(resp. $\mathrm{gcd}(N,b)$-points above $x=1$).
Let $u$ be a path from $P_1$ to $P_2$ above the real interval $x\in[0,1]$.
It defines a homology cycle $u\in H_1(X_t,\{P_1,P_2\};\Z)$ with boundary.
Put $d_1:=\mathrm{gcd}(N,a)$, $d_2:=\mathrm{gcd}(N,b)$.
Since $\sigma^{d_1}P_1=P_1$ and $\sigma^{d_2}P_2=P_2$ for $\sigma\in\mu_N$
an automorphism,
one has a homology cycle
\begin{equation}\label{HG-sect1-eq2-p}
\delta:=(1-\sigma^{d_1})(1-\sigma^{d_2})u\in H_1(X_t,\Z).
\end{equation}
By an integral expression of Gauss hypergeometric functions 
(e.g. \cite{Bailey} p.4, 1.5 or \cite{slater} p.20, (1.6.6)), one has
\begin{align}
\int_\delta\omega_n
&=(1-\zeta^{-nd_1})(1-\zeta^{-nd_2})\int_0^1\omega_n\\
&=(1-\zeta^{-nd_1})(1-\zeta^{-nd_2})
B(\alpha_n,\beta_n){}_2F_1(\alpha_n,\beta_n,\alpha_n+\beta_n;t),
\label{HG-sect1-eq2}
\end{align}
where $\sigma(y)=\zeta y$ and 
\[
\alpha_n:=\left\{\frac{-an}{N}\right\},\quad 
\beta_n:=\left\{\frac{-bn}{N}\right\}.
\]
This shows that the monodromy on the 2-dimensional $H_1(X_t,\C)(\chi)$ is isomorphic
to the monodromy of the hypergeometric equation
\[
D_t(D_t+\alpha_n+\beta_n-1)-t(D_t+\alpha_n)(D_t+\beta_n),\quad D_t:=t\frac{d}{dt}
\] 
with the Riemann scheme
\begin{equation}\label{HG-sect1-eq3}
\left\{
\begin{matrix}
t=0&t=1&t=\infty\\
0&0&\alpha_n\\
1-\alpha_n-\beta_n&0&\beta_n
\end{matrix}
\right\}
\end{equation}
In particular, the monodromy is irreducible as $\alpha_n,\beta_n\not\in\Z$. 

\subsection{Hodge numbers}
Let $f:X\to \P^1$ be a HG fibration with multiplication by $(R_0,e_0)$.
Following \cite{a-o-3} \S 4.1, 
we consider motivic sheaves $\cM$ and $\cH$ which are defined in the following way.
Let $S:=\bA^1_{\ol\Q}\setminus\{0,1\}$ be defined over $\ol\Q$
with coordinate $\lambda$.
Let $\P^1_S:=\P^1\times S$ and denote the coordinates by $(t,\lambda)$.
Put $\P^1_S\supset\cU:=(\bA^1_{\ol\Q}\setminus\{0,1\}\times S)\setminus\Delta$
where $\Delta$ is the diagonal subscheme.
Let $l\geq 1$ be an integer. 
Let $\pi:\P^1_S\to\P^1_S$ be a morphism over $S$ given by $(t,\lambda)\mapsto 
(\lambda-t^l,\lambda)$.
Then we define
\[
\cM:=\pi_*\Q\ot \pr_1^*R^1f_*\Q|_{\mathscr U},\quad\pr_1:
\P^1_S=\P^1\times S\to\P^1
\]
a variation of Hodge-de Rham structures (abbreviated VHdR) on $\cU$ and
\[
\cH:=R^1\pr_{2*}\cM
,\quad \pr_2:\cU\to S
\] 
a variation of mixed Hodge-de Rham structures (abbreviated VMHdR) on $S$.
The weights of $\cH$ are at most $2,3,4$, and hence 
there is an exact sequence
\begin{equation}\label{HG-sect1-eq4}
0\lra W_2\cH\lra \cH\lra \cH/W_2\lra0
\end{equation}
of VMHdR structures on $S$.

Let $\mu_l$ be the group of $l$-th roots of unity which acts on $\pi_*\Q$ in a natural way.
Then $\cM$ has multiplication by the group ring $R:=R_0[\mu_l]$.
Let $e:R\to E$ be a projection onto a number field $E$ such that $\ker(e)\supset \ker(e_0)$.
There is unique embedding $E_0\hra E$ making the diagram
\[
\xymatrix{
R_0\ar[d]\ar[r]^{e_0}&E_0\ar[d]\\
R\ar[r]^e&E\\
}
\]
commutative.

For $\chi:R\to\ol\Q$ factoring through $E$, we denote by $V(\chi)$ the $\chi$-part
which is defined to be the subspace on which $r\in R$ acts by
multiplication by $\chi(r)$.

\begin{thm}\label{HI-prop}
Let $T_p$ denotes the local monodromy on $R^1f_*\ol\Q(\chi)$ at $t=p$.
Let $\alpha_j^\chi$ (resp, $\beta_j^\chi$) for $j=1,2$ be rational numbers
such that  $e^{2\pi i\alpha_j^\chi}$ (resp, $e^{2\pi i\beta_j^\chi}$) are eigenvalues of 
$T_0$ (resp. $T_\infty$).
Let $k$ be an integer such that $\chi(\zeta_l)=\zeta_l^k$ for $\zeta_l\in\mu_l$.
Suppose that $k/l,-k/l+\beta^\chi_j\not\in\Z$ and $\alpha_1^\chi\in\Z$.
Write $h_\chi^{p,2-p}:=\dim_{\ol\Q} \Gr^p_FW_2\cH(\chi)$.
Put
\[
d_\chi:=
2\{-k/l\}+\sum_{i=1}^2\{\beta^\chi_i\}-\{\beta^\chi_i-k/l\}.
\]
Then 
\[
(h^{2,0}_\chi,h^{1,1}_\chi,h^{0,2}_\chi)=
\begin{cases}
(1,1,0)&\mbox{if }d_\chi=2\\
(0,2,0)&\mbox{if }d_\chi=1\\
(0,1,1)&\mbox{if }d_\chi=0.
\end{cases}
\]
\end{thm}
Note that $d_\chi$ takes values only in $0,1$ or $2$. Indeed
\[
d_\chi=
\overbrace{\{\beta^\chi_1\}+\{-k/l\}-\{\beta^\chi_1-k/l\}}^{\delta_1}
+\overbrace{\{\beta^\chi_2\}+\{-k/l\}-\{\beta^\chi_2-k/l\}}^{\delta_2}
\]
and each $\delta_i$ takes values $0$ or $1$. 
\begin{proof}
We first note that $\dim_E W_2\cH(e)=2$ (\cite{a-o-3} \S 4.3).
We employ two results from \cite{asakura-fresan} and \cite{fedorov} respectively.
First of all, it follows from \cite{asakura-fresan} Theorem 4.2
that one has 
the Hodge numbers of the determinant
$D:=\mathrm{det}_EW_2\cH(e)=\bigwedge^2_EW_2\cH(e)$.
The result is
\[
(D^{4,0}_\chi,D^{3,1}_\chi,D^{2,2}_\chi,D^{1,3}_\chi,D^{0,4}_\chi)=
\begin{cases}
(0,1,0,0,0)&\mbox{if }d_\chi=2\\
(0,0,1,0,0)&\mbox{if }d_\chi=1\\
(0,0,0,1,0)&\mbox{if }d_\chi=0
\end{cases}
\]
where we put $D^{p,4-p}_\chi:=\dim\Gr^p_FD(\chi)$.
Since $D_\chi^{p,4-p}=1$ $\Leftrightarrow$ $2h^{2,0}_\chi+h^{1,1}_\chi=p$, this implies
\[
(h^{2,0}_\chi,h^{1,1}_\chi,h^{0,2}_\chi)=
\begin{cases}
(1,1,0)&\mbox{if }d_\chi=2\\
(0,2,0)\mbox{ or }(1,0,1)&\mbox{if }d_\chi=1\\
(0,1,1)&\mbox{if }d_\chi=0
\end{cases}
\]
which completes the proof in the case $d_\chi\ne1$.
Suppose $d_\chi=1$. 
We want to show that $(h^{2,0}_\chi,h^{1,1}_\chi,h^{0,2}_\chi)=(1,0,1)$
cannot happen.
By \cite{a-o-3} Theorem 5.8, the underlying connection
of $W_2\cH(\chi)$ is defined by the hypergeometric differential operator as in loc.cit.
One can apply the main theorem in \cite{fedorov} and then
the possible triplets of the Hodge numbers are at most $(2,0,0),(0,2,0),(0,0,2)$.
In particular the case $(h^{2,0}_\chi,h^{1,1}_\chi,h^{0,2}_\chi)=(1,0,1)$ is excluded.
This completes the proof in case $d_\chi=1$.
\end{proof}
\begin{rem}
In the latter half of the proof of Theorem \ref{HI-prop},
there is an alternative discussion without using the main theorem of \cite{fedorov}.
Let $\pi_0:\P^1\to \P^1$ be a map given by $t\mapsto -t^l$.
Let $\cM_0:=\pi_{0*}\Q\ot R^1f_*\Q$ be a VHdR on $\P^1\setminus\{0,1,\infty\}$.
Put $H_0:=H^1(\P^1\setminus\{0,1,\infty\},\cM_0)$.
Let $\psi_{\lambda=0}$ denotes the nearby cycle functor.
Then one can construct an injection
\[
\xymatrix{
E\cong W_2H_0(e)\ar@{^{(}->}[r]& \psi_{\lambda=0}W_2\cH(e)
}
\]
of mixed Hodge-de Rham structures.
The cohomology group $W_2H_0(e)$ is studied in detail in \cite{a-o}.
In particular, if $d_\chi=1$, then the Hodge type of $W_2H_0(\chi)$ is $(1,1)$.
Hence $h^{1,1}_\chi>0$ by the above injection, which excludes the case
$(h^{2,0}_\chi,h^{1,1}_\chi,h^{0,2}_\chi)=(1,0,1)$.
\end{rem}

\begin{cor}\label{HI}
$W_2\cH(e)$ is a Tate VHdR of type $(1,1)$ if and only if
$d_\chi=1$ for all $\chi:R\to\ol\Q$, equivalently
\[
2\{-sk_0/l\}+\sum_{i=1}^2\{s\beta^{\chi_0}_i\}-\{s(\beta^{\chi_0}_i-k_0/l)\}=1
\]
\[
\Longleftrightarrow\quad 
\{s\beta^{\chi_0}_1\}<\{sk_0/l\}<\{s\beta^{\chi_0}_2\}\mbox{ or }
\{s\beta^{\chi_0}_2\}<\{sk_0/l\}<\{s\beta^{\chi_0}_1\}
\]
for $\forall s\in \hat{\Z}^\times$
where $\chi_0$ is a fixed one and $\beta_j^{\chi_0}, k_0$ are the rational numbers
arising from $\chi_0$.
\end{cor}
\subsection{Beilinson Regulator}
Let $\psi_{t=1}$ be the nearby cycle functor along the function $t-1$ on $\cU$, and put
\[
C:=\Gr^W_2\psi_{t=1}\cM\cong \pi_*\Q|_{\{1\}\times S}\ot (\Gr^W_2\psi_{t=1}R^1f_*\Q)
\]
a VHdR on $S$. Then there is a natural embedding
\[
C\ot\Q(-1)\lra \cH/W_2
\]
and it gives an extension
\begin{equation}\label{HG-sect1-eq5}
\xymatrix{
0\ar[r] &W_2\cH(e)\ar[r]& \cH'(e)\ar[r]& C(e)\ot\Q(-1)\ar[r]&0
}
\end{equation}
of VMHdR with multiplication by $E$ which is induced from \eqref{HG-sect1-eq4}.
Note $C(e)$ is one-dimensional over $E$ and endowed with Hodge type $(1,1)$
by {\bf (c)} in \S \ref{HG-sect1}.

In \cite{a-o-3} \S 5 we discussed the extension data of \eqref{HG-sect1-eq5}.
More precisely let $\O^\zar$ be 
the Zariski sheaf of polynomial functions (with coefficients in $\ol\Q$) on 
$S=\bA^1_{\ol\Q}\setminus\{0,1\}$ with coordinate $\lambda$.
Let $\O^{an}$ be the sheaf of analytic functions on $S^{an}=\C^{an}\setminus\{0,1\}$.
Let $a:S^{an}\to S^\zar$
be the canonical morphism from the analytic site to the Zariski site.
Set
\[
\cJ:=\Coker[a^{-1}F^2W_2\cH_\dR\op\iota^{-1}W_2\cH_B\to \O^{an}\ot_{a^{-1}\O^\zar}a^{-1}W_2\cH_\dR]
\]
a sheef on the analytic site $\C^{an}\setminus\{0,1\}$.
Let $h:\wt{S}\to S$ be a generically finite and dominant map
such that $\sqrt[l]{\lambda-1}\in \ol\Q(\wt{S})$. 
Then $h^*C(e)$ is a constant VHdR of type $(1,1)$.
The connecting homomorphism arising from \eqref{HG-sect1-eq5} gives a map
\begin{equation}\label{HG-sect1-eq6}
\rho:h^*C(e)\ot\Q(1)\lra \vg(\wt{S}^{an},h^*\cJ(e))
\end{equation}
(see \cite{a-o-3} \S 5.2 for the detail).
This agrees with the {\it Beilinson regulator map} on the
motivic cohomology supported on singular fibers.
Let $\pi:\P^1_{\wt{S}}:=\P^1\times_{\ol\Q} \wt{S}\to \P^1$ be given by 
$(s,\lambda')\mapsto
h(\lambda')-s^l$. Let
\[
\xymatrix{
X_{\wt{S}}\ar[d]_{g}\ar[r]^{i\qquad}\ar[rd]^{f_{\wt{S}}}&\P^1_{\wt{S}}
\times_{\P^1}X\ar[r]\ar[d]&X\ar[d]^f\\
\wt S&\P^1_{\wt{S}}\ar[r]^\pi\ar[l]_p&\P^1
}
\]
with $i$ desingularization and $p$ the 2nd projection.
Let
\[
\reg:H^3_\cM(X_{\wt{S}},\Q(2))\lra H^3_\cD(X_{\wt{S}},\Q(2))
=\Ext^3_{\mathrm{MHM}(X_{\wt S})}(\Q,\Q(2))
\]
be the Beilinson regulator map where $\mathrm{MHM}(\wt{S})$ denotes the category
of mixed Hodge modules on $\wt S$.
There is the canonical surjective map
\[
\Ext^3_{\mathrm{MHM}(X_{\wt S})}(\Q,\Q(2))
\lra \Ext^1_{\mathrm{VMHdR}(\wt S)}(\Q,R^2g_*\Q(2)).
\]
Let $U_{\wt S}\subset \P^1_{\wt S}$ be a Zariski open set on which
$f_{\wt S}$ is smooth and projective.
Put 
\[
H^3_\cM(X_{\wt{S}},\Q(2))_0:=\ker[
H^3_\cM(X_{\wt{S}},\Q(2))\lra H^3_\cM(f^{-1}_{\wt{S}}(U_{\wt S}),\Q(2))]
\]
and $(R^2g_*\Q(2))_0:=\Ker[R^2g_*\Q(2)\to p_*(R^2(f_{\wt S})_*\Q(2)|_{U_{\wt S}})]$.
Then there is the canonical surjective map $(R^2g_*\Q(2))_0\to h^*W_2\cH(2)$ and 
the Beilinson regulator map induces
\[
\reg_0:H^3_\cM(X_{\wt{S}},\Q(2))_0\lra 
\Ext^1_{\mathrm{VMHdR}(\wt S)}(\Q,h^*W_2\cH(2))\lra \vg(\wt{S}^{an},h^*\cJ).
\]
The compatibility with \eqref{HG-sect1-eq6} is given by the commutativity of
a diagram
\begin{equation}\label{HG-sect1-eq4-b1}
\xymatrix{
H^3_{\cM,D_{\wt{S}}}(X_{\wt{S}},\Q(2))\ar[r]\ar[d]&h^*C(e)\ot\Q(1)\ar[d]^\rho\\
H^3_\cM(X_{\wt{S}},\Q(2))_0\ar[r]^{\reg_0}
&\vg(\wt{S}^{an},h^*\cJ)
}
\end{equation}
where $D_{\wt S}:=X_{\wt S}\setminus U_{\wt S}$.

\subsection{Regulator Formula for HG fibrations of Gauss type}\label{HG-sect3}
One of the main results in \cite{a-o-3} (which we call {\it regulator formula})
is an explicit description of the map $\rho$ in \eqref{HG-sect1-eq6}.
Here we apply \cite{a-o-3} Theorem 5.9 (=a precise version of regulator formula)
to the case that $f$ is a HG fibration of Gauss type (see Definition \ref{gauss.type}).

\medskip

Let $f:X\to \P^1$ be a HG fibration of Gauss type with multiplication
by $(R_0:=\Q[\mu_N],e_0)$ as in Definition \ref{gauss.type}.
Let $\chi:E_0\to\ol\Q$ be a homomorphism such that $\sigma(\zeta)=\zeta^{-n}$. 
Recall from \S \ref{HG-sect2} that $F^1H^1_\dR(X_t)(\chi)$ is one dimensional
and spanned by a 1-form
\[
\omega_n:=\frac{x^{a_n}(1-x)^{b_n}(1-tx)^{c_n}}{y^n}dx,
\]
\[
a_n:=\lfloor\frac{an}{N}\rfloor,\,
b_n:=\lfloor\frac{bn}{N}\rfloor,\,
c_n:=\lfloor\frac{Nn-bn}{N}\rfloor=n-b_n-1
\]
where $n\in\{1,2,\ldots,N-1\}$ such that $\chi(\zeta)=\zeta^{-n}$ 
for $\forall\zeta\in\mu_N$.
\begin{lem}\label{HG-sect3-lem1}
Let $D_0,D_1$ be the reduced singular fibers over $t=0,1$. We assume that $D_0+D_1$ 
is a NCD.
Then $t\omega_n\in \vg(\P^1\setminus\{\infty\},f_*\Omega^1_{X/\P^1}(\log D_0+D_1))$.
\end{lem}
\begin{proof}
Put $S=\P^1\setminus\{0,1,\infty\}$ and $U=f^{-1}(S)$. 
Let $\cH:=H^1_\dR(U/S)$ be the bundle and $\nabla:\cH\to\Omega^1_S\ot\cH$ the
Gauss-Manin connection. 
Let $D_\infty$ be the reduced singular fibers over $t=\infty$ and assume that it is a NCD.
Put $T:=\{0,1,\infty\}$.
Recall that the sheaf $\Omega^1_{X/\P^1}(\log D)$ ($D:=D_0+D_1+D_\infty)$ is
defined by the exact sequence
\[
0\lra f^*\Omega^1_{\P^1}(\log T)\lra \Omega^1_X(\log D)\lra
\Omega^1_{X/\P^1}(\log D)\lra0.
\]
Let $\cH_e$ be Deligne's canonical extension over $\P^1$. This is characterized as
a subbundle $\cH_e\subset j_*\cH$ $(j:S\hra\P^1$) which satisfies
\begin{itemize}
\item
$\nabla$ has at most log poles, $\nabla:\cH_e\to\Omega^1_{\P^1}(\log(0+1+\infty))
\ot\cH_e$,
\item
The eigenvalues of residue $\Res(\nabla)$ at $t=0,1,\infty$ belong to $[0,1)$. 
\end{itemize}
Then there is an isomorphism
\[
\cH_e\cong R^1f_*\Omega^\bullet_{X/\P^1}(\log D)
\]
(\cite{steenbrink} 2.20) and $F^1\cH_e:=\cH_e\cap 
j_*F^1\cH\cong f_*\Omega^1_{X/\P^1}(\log D)$
(loc.cit. 4.20 (ii)).
Hence the desired assertion is equivalent to
\begin{equation}\label{HG-sect3-lem1-eq1}
t\omega_n\in \vg(\P^1\setminus\{\infty\},\cH_e).
\end{equation}
To show this, we give a local frame of $\cH_e$ at $t=0,1$ explicitly.
Let
\[
\eta_n:=\frac{x^{a_n}(1-x)^{b_n+1}(1-tx)^{c_n}}{y^n}dx,
\]
and put
\[\beta^\chi_1:=\left\{\frac{-an}{N}\right\},\quad
\beta^\chi_2:=\left\{\frac{-bn}{N}\right\}.
\]
Recall from \eqref{HG-sect1-eq2-p} 
a homology cycle $\delta:=(1-\sigma^{d_1})(1-\sigma^{d_2})u\in H_1(X_t,\Z)$.
Then 
\begin{equation}\label{HG-sect3-lem1-eq2}
\int_\delta\omega_n=
(1-\zeta^{-nd_1})(1-\zeta^{-nd_2})B(\beta^\chi_1,\beta^\chi_2)
F(\beta^\chi_1,\beta^\chi_2,\beta^\chi_1+\beta^\chi_2;t),
\end{equation}
\begin{equation}\label{HG-sect3-lem1-eq3}
\int_\delta\eta_n=
(1-\zeta^{-nd_1})(1-\zeta^{-nd_2})B(\beta^\chi_1,\beta^\chi_2+1)
F(\beta^\chi_1,\beta^\chi_2,1+\beta^\chi_1+\beta^\chi_2;t).
\end{equation}
This shows that $\omega_n$ and $\eta_n$ are basis of the $\chi$-part $\cH(\chi)$
of the bundle (over a Zariski open set of $\P^1\setminus\{0,1,\infty\}$).
Denote by $\cH(\chi)^*$ the dual connection, and by
$\{\omega^*_n,\eta^*_n\}$ the dual basis. Then 
\[
\left(\int_\delta\omega_n\right)\omega^*_n
+\left(\int_\delta\eta_n\right)\eta^*_n
\]
is annihilated by the dual connection, and hence
\begin{equation}\label{HG-sect3-lem1-eq4}
d\left(\int_\delta\omega_n\right)\omega^*_n
+d\left(\int_\delta\eta_n\right)\eta^*_n
+\left(\int_\delta\omega_n\right)\nabla(\omega^*_n)
+\left(\int_\delta\eta_n\right)\nabla(\eta^*_n)=0.
\end{equation}
Now
\eqref{HG-sect3-lem1-eq2},
\eqref{HG-sect3-lem1-eq3},
\eqref{HG-sect3-lem1-eq4} together with the formulas
\[
(1-t)\frac{d}{dt}F(a,b,a+b;t)=\frac{ab}{a+b}F(a,b,a+b+1;t),
\]
\[
t\frac{d}{dt}F(a,b,a+b+1;t)=(a+b)(F(a,b,a+b;t)-F(a,b,a+b+1;t))
\]
imply
\[
(\nabla(\omega^*_n),\nabla(\eta^*_n))=dt\ot(\omega^*_n,\eta^*_n)\begin{pmatrix}
0&-\beta^\chi_1/(1-t)\\
-\beta^\chi_2/t&(\beta^\chi_1+\beta^\chi_2)/t
\end{pmatrix}
\]
\[
\Longleftrightarrow\quad(\nabla(\omega_n),\nabla(\eta_n))=dt\ot(\omega_n,\eta_n)\begin{pmatrix}
0&\beta^\chi_2/t\\
\beta^\chi_1/(1-t)&-(\beta^\chi_1+\beta^\chi_2)/t
\end{pmatrix}.
\]
Then it is an elementary linear algebra to compute
local frames of $\cH_e$ :
\[
\cH_e(\chi)|_{t=0}=\begin{cases}
\langle \omega_n,t(\beta_2^\chi\omega_n+
(\beta^\chi_1+\beta^\chi_1)\eta_n)\rangle &\beta^\chi_1+\beta^\chi_2\leq 1\\
\langle t\omega_n,(\beta^\chi_1+\beta^\chi_2-1)\omega_n+t\beta^\chi_1
\eta_n\rangle &\beta^\chi_1+\beta^\chi_2> 1
\end{cases}
\]
\[
\cH_e(\chi)|_{t=1}=
\langle \omega_n,\eta_n\rangle.
\]
Now \eqref{HG-sect3-lem1-eq1} is immediate. 
\end{proof}

Let $e_0:\mu_N \to E_0^\times$ be an injective homomorphism.
Then the condition in Lemma \ref{gauss.cond} is satisfied.
Let $e:R:=\Q[\mu_l,\mu_N]\to E$ be a projection such that $\ker(e)\supset\ker(e_0)$.
Let $\chi:R\to \ol\Q$ be a homomorphism factoring through $e$.
Fix integers $k,n$ such that
\[
\chi(\zeta_1,\zeta_2)=\zeta_1^k\zeta_2^n,\quad \forall(\zeta_1,\zeta_2)\in\mu_l\times\mu_N.
\]
Note $\mathrm{gcd}(n,N)=1$ as $e_0:\mu_N \to E_0^\times$ is injective.
Let
\begin{equation}\label{HG-sect3-eq1}
\beta_1^\chi:=\left\{\frac{-na}{N}\right\},\quad
\beta_2^\chi:=\left\{\frac{-nb}{N}\right\},\quad
\alpha_1^\chi:=0,\quad
\alpha^\chi_2:=1-\beta^\chi_1-\beta^\chi_2
\end{equation}
which do not depend on the choice of $n$.
Then $e^{2\pi i\alpha_j^\chi}$ (resp. $e^{2\pi i\beta_j^\chi}$)
are eigenvalues of the local monodromy $T_0$ at $t=0$ (resp. $T_\infty$ at $t=\infty$)
on $R^1f_*\C(\chi)\cong\C^2$ (see \eqref{HG-sect1-eq3}).
The relative 1-form $\omega:=t\omega_n$
satisfies the conditions {\bf P1}, {\bf P2} in \cite{a-o-3} \S 4.5 :
\begin{enumerate}
\item[{\bf P1}]
$\int_{\gamma_t} \omega$ $(\gamma_t\in H_1(X_t))$
is spanned by $tF(\beta^\chi_1,\beta^\chi_2,1;1-t)$ and
$tF(\beta^\chi_1,\beta^\chi_2,\beta^\chi_1+\beta^\chi_2;t)$.
(This follows from \eqref{HG-sect1-eq2}).
\item[{\bf P2}]
$\omega\in \vg(\P^1\setminus\{\infty\},f_*\Omega^1_{X/\P^1}(\log D))$.
(This is Lemma \ref{HG-sect3-lem1}).
\end{enumerate}
We thus can apply the {\it regulator formula} (\cite{a-o-3} Thm.5.9).
In our particular case, it is stated as follows (the notation is slightly changed 
for the use in below).
\begin{thm}\label{a-o-log}
Let $e_0,e,\chi$ be as above, and let $\alpha_i^\chi$, $\beta_j^\chi$ be as in 
\eqref{HG-sect3-eq1}.
Assume that
$k/l,k/l-\beta_1^\chi,k/l-\beta_2^\chi,k/l-\beta_1^\chi-\beta_2^\chi\not\in\Z$.
Put
\[
\cF_1(\lambda):=(1-\lambda)^{k/l-1}
{}_3F_2\left({1,1,1-k/l\atop 2-\beta_1^\chi,2-\beta_2^\chi};(1-\lambda)^{-1}\right),
\]
\[
\cF_2(\lambda):=(1-\lambda)^{k/l-1}
{}_3F_2\left({1,1,2-k/l\atop 2-\beta_1^\chi,2-\beta_2^\chi};(1-\lambda)^{-1}\right).
\]
Let $\rho({}^t\chi)$ be the ${}^t\chi$-part of  the map $\rho$ in \eqref{HG-sect1-eq6}.
Let
\[
\rho({}^t\chi)=(\phi_1(\lambda),\phi_2(\lambda))\in (\O^{an})^{\op2}
\cong \O^{an}\ot W_2\cH_\dR({}^t\chi)
\] 
be a local lifting where the above isomorphism is with respect to $\ol\Q$-frame
of $W_2\cH_\dR({}^t\chi)$.
Define 
rational functions $E_i^{(r)}=E_i^{(r)}(\lambda)\in\Q(\lambda)$ for $r\in\Z_{\geq-1}$
in the following way.
Write $a:=2-\beta_1^\chi$, $b:=2-\beta_2^\chi$. Put
\[
A(s):=\frac{s(a+b+2s-3-s(1-\lambda)^{-1})}{(a+s-1)(b+s-1)},\quad
B(s):=\frac{s(1-s)(1-(1-\lambda)^{-1})}{(a+s-1)(b+s-1)}
\]
Define $C_i(s)$ and $D_i(s)$ by
\[
\begin{pmatrix}
C_{i+1}(s)\\
D_{i+1}(s)
\end{pmatrix}
=
\begin{pmatrix}
A(s)&1\\
B(s)&0
\end{pmatrix}
\begin{pmatrix}
C_i(s+1)\\
D_i(s+1)
\end{pmatrix},
\quad \begin{pmatrix}
C_{-1}(s)\\D_{-1}(s)
\end{pmatrix}
:=\begin{pmatrix}0\\1\end{pmatrix}.
\]
Then $E_i^{(r)}$ is given by
\begin{equation}\label{HG-sect3-eq2}
E_1^{(r)}=\lambda C_r(k/l)+(1-\lambda)C_{r+1}(k/l),\quad 
E_2^{(r)}=\lambda D_r(k/l)+(1-\lambda)D_{r+1}(k/l).
\end{equation}
Then for infinitely many integers $r>0$, we have
\begin{align*}
\phi_1(\lambda)&\equiv C_1(1-\lambda)^r
[E_1^{(r)}(\lambda)\cF_1(\lambda)+
E_2^{(r)}(\lambda)\cF_2(\lambda)],\\
\phi_2(\lambda)&\equiv C_2(1-\lambda)^{r-1}
[E_1^{(r-1)}(\lambda)\cF_1(\lambda)+
E_2^{(r-1)}(\lambda)\cF_2(\lambda)]
\end{align*}
modulo $\ol{\Q(\lambda)}$ with some $C_1,C_2\in\ol\Q^\times$.
\end{thm}
We note that $(N,l,k,n,a,b)$ in Theorem \ref{a-o-log} can run over the set of all
pairs of integers satisfiyng
\begin{itemize}
\item $0<a,b<N$, $\gcd(N,a,b)=1$ and $\gcd(n,N)=1$,
\item
$k/l,k/l-\beta_1^\chi,k/l-\beta_2^\chi,k/l-\beta_1^\chi-\beta_2^\chi\not\in\Z$
(see \eqref{HG-sect3-eq1} for definition of $\beta_j^\chi$).
\end{itemize}

\section{Proof of Main Theorem}\label{proof-sect}
We are now in a position to prove Theorem \ref{log} (Log Formula).

\medskip

There are formulas
\begin{align*}
(b_1-1){}_3F_2\left({a_1,a_2,a_3\atop b_1-1,b_2};x\right)
&=\left(b_1-1+x\frac{d}{dx}\right)
{}_3F_2\left({a_1,a_2,a_3\atop b_1,b_2};x\right),\\
a_1\cdot\,{}_3F_2\left({a_1+1,a_2,a_3\atop b_1,b_2};x\right)
&=\left(a_1+x\frac{d}{dx}
\right)
{}_3F_2\left({a_1,a_2,a_3\atop b_1,b_2};x\right),
\end{align*}
\begin{align*}
(a_2-b_1)(a_1-b_1)(a_3-b_1){}_3F_2\left({a_1,a_2,a_3\atop b_1+1,b_2};x\right)
&=\theta_1\left({}_3F_2\left({a_1,a_2,a_3\atop b_1,b_2};x\right)\right),\\
(a_1-b_1)(a_1-b_2){}_3F_2\left({a_1-1,a_2,a_3\atop b_1,b_2};x\right)
&=\theta_2\left({}_3F_2\left({a_1,a_2,a_3\atop b_1,b_2};x\right)\right),
\end{align*}
where
\begin{align*}
\theta_1&:=-a_1a_2a_3+(a_2-b_1)(a_1-b_1)(a_3-b_1)\\
&\qquad +b_1(b_2+(b_1-a_1-a_2-a_3-1)x)\frac{d}{dx}+b_1(x-x^2)\frac{d^2}{dx^2}
\\
\theta_2&:=
(a_1-b_1)(a_1-b_2)-a_2a_3x\\
&\qquad+((b_1+b_2-a_1)-(a_2+a_3+1)x)x\frac{d}{dx}
 +(1-x)x^2\frac{d^2}{dx^2}.
\end{align*}
Therefore if one can show the log formula for ${}_3F_2\left({1,1,q\atop a,~b};x\right)$
then one immediately has the log formula for ${}_3F_2\left({n_1,n_2,q+n_3\atop a+n_4,~b+n_5};x\right)$
for arbitrary integers $n_1,n_2>0$ and $n_3,n_4,n_5\in\Z$.

\medskip

We keep the setting and the notation in \S \ref{HG-sect3}.
Suppose that
\begin{equation}\label{reg-log-1}
1=2\{-sk/l\}+\sum_{i=1}^2\{s\beta^\chi_2\}-\{s(\beta^\chi_i-k/l)\},\quad \forall s\in\hat{\Z}^\times.
\end{equation}
Then it follows from Corollary \ref{HI} 
that $W_2\cH(e)$ is a Tate HdR structure of type $(1,1)$.
Let us look at the map $\rho({}^t\chi)$ in Theorem \ref{a-o-log}.
This turns out to be the Beilinson regulator by the diagram \eqref{HG-sect1-eq4-b1}.
Since $W_2\cH(e)$ is Tate, it is generated by the divisor classes of 
the geometric generic fiber $X_{\ol\eta}$ of $f_{\wt{S}}$.
This implies that the image of $\reg$ in \eqref{HG-sect1-eq4-b1} is generated by
the images of $H^1_{\cM}(\wt{D}_i,\Q(1))$ where $D_i$ runs over
the generators of the Neron-Severi group $\NS(X_{\ol\eta})\ot\Q$ and 
$\wt{D}_i\to D_i$ is the desingularization.
As is well-known, $H^1_{\cM}(\wt{D}_i,\Q(1))\cong \ol\eta^\times\ot\Q$ as $\wt{D}_i$
is smooth projective, and the Beilinson regulator on it is given by the logarithmic function.
Therefore we have
\begin{equation}\label{reg-log-2}
\phi_1(\lambda),\,\phi_2(\lambda)\in \ol{\Q(\lambda)}+\ol{\Q(\lambda)}\log\ol{\Q(\lambda)}^\times.
\end{equation}
We now apply Theorem \ref{a-o-log}.
If one can show
\[
\begin{vmatrix}
E_1^{(r)}& E_2^{(r)}\\
E_1^{(r-1)}&E_2^{(r-1)}
\end{vmatrix}\ne0
\]
for almost all $r>0$, then this implies
$\cF_i(\lambda)\in \ol{\Q(\lambda)}+\ol{\Q(\lambda)}\log\ol{\Q(\lambda)}^\times$,
which finishes the proof of Theorem \ref{log}.
To do this, recall \eqref{HG-sect3-eq2}.
Letting 
\[
E_1^{(r)}(s):=\lambda C_r(s)+(1-\lambda)C_{r+1}(s),\quad 
E_2^{(r)}(s):=\lambda D_r(s)+(1-\lambda)D_{r+1}(s),
\]
we want to show
\begin{equation}\label{reg-log-3}
\begin{vmatrix}
E_1^{(r)}(k/l)& E_2^{(r)}(k/l)\\
E_1^{(r-1)}(k/l)&E_2^{(r-1)}(k/l)
\end{vmatrix}\ne0
\end{equation}
for almost all $r>0$.
Since
\[
\begin{pmatrix}
E_1^{(r+1)}(s)& E_1^{(r)}(s)\\
E_2^{(r+1)}(s)&E_2^{(r)}(s)
\end{pmatrix}=
\begin{pmatrix}
A(s)&1\\
B(s)&0
\end{pmatrix}
\begin{pmatrix}
E_1^{(r)}(s+1)& E_1^{(r-1)}(s+1)\\
E_2^{(r)}(s+1)&E_2^{(r-1)}(s+1)
\end{pmatrix}
\]
\eqref{reg-log-3} is reduced to show
\[
\begin{vmatrix}
E_1^{(0)}(k/l+r)& E_2^{(0)}(k/l+r)\\
E_1^{(-1)}(k/l+r)&E_2^{(-1)}(k/l+r)
\end{vmatrix}\ne0
\]
for any integers $r$.
However this follows by 
\begin{align*}
\begin{vmatrix}
E_1^{(0)}(s)& E_2^{(0)}(s)\\
E_1^{(-1)}(s)&E_2^{(-1)}(s)
\end{vmatrix}&=
\begin{vmatrix}
\lambda+(1-\lambda)A(s)&
(1-\lambda) B(s)\\
1-\lambda&\lambda
\end{vmatrix}\\
&=\lambda\frac{(a-1)(b-1)\lambda+s(a+b-2)}{(s+a-1)(s+b-1)},\quad (a:=2-\beta^\chi_1,\,
b:=2-\beta^\chi_2)
\end{align*}
and the fact $\beta_i^\chi\not\in\Z$ (see \eqref{HG-sect3-eq1})
and $k/l-\beta_i^\chi\not\in\Z$ as is assumed.
This completes the proof of Theorem \ref{log}.

\end{document}